\documentclass[11pt]{amsart}
\usepackage{amsmath, amsthm, amscd, amssymb,amsfonts}
 \usepackage{color}
 \usepackage{mathrsfs}
 \usepackage[all]{xy}
 \usepackage{hyperref}
\usepackage{amsmath, amsthm, amscd, amssymb,amsfonts}
\usepackage[toc,page]{appendix}
\usepackage{color}
\usepackage{mathrsfs}

\numberwithin{equation}{section}

\newtheorem{thm}{Theorem}[section]
\newtheorem*{thmA}{Theorem A}
\newtheorem*{thmB}{Theorem B}
\newtheorem*{thmC}{Theorem C}

\newtheorem{problem}[thm]{Problem}

\newtheorem*{cor*}{Corollary}
\newtheorem{lem}[thm]{Lemma}

\newtheorem*{conj*}{Conjecture}
\newtheorem*{prob*}{Problem}

\theoremstyle{definition}
\newtheorem{defn}[thm]{Definition}

\theoremstyle{remark}
\newtheorem{rem}[thm]{Remark}

\newcommand{\ds}{\displaystyle}

\begin{document}
\title{On Passage to over-groups of finite indices of The Farrell-Jones conjecture}

\author[Kun Wang]{Kun Wang}
\address{Department of Mathematics,
Vanderbilt University,
1326 Stevenson Center,
Nashville, TN 37240}
\email{kun.wang@vanderbilt.edu}

\begin{abstract}
We use the controlled algebra approach to study the problem that whether the 
Farrell-Jones conjecture is closed under passage to over-groups of finite indices. Our study shows that this problem is closely related to a general problem in algebraic $K$- and $L$-theories. We use induction theory to study this general problem.  This requires an extension of the classical induction theorem for $K$- and $L$- theories of finite groups with coefficients in rings to with twisted coefficients in additive categories. This extension is well-known to experts, but a detailed proof does not exist in the literature. We carry out a detailed proof.  This extended induction theorem enables us to make some reductions for the general problem, and therefore for the finite index problem of the Farrell-Jones conjecture. 

\end{abstract}

\maketitle
\section{Introduction}
The algebraic $K$- and $L$-groups $K_n(\mathbb Z[G]), L_n(\mathbb Z[G]), n\in\mathbb Z$ of the integral group ring $\mathbb Z[G]$, where $G$ is the fundamental group of a topological space, contain important information about the underlying space. The key tool in understanding these groups is the Farrell-Jones conjecture (FJC for short). Roughly speaking, the conjecture states that for every group $G$, the $K$- and $L$- theories of $\mathbb Z[G]$ is determined by those of its virtually cyclic subgroups and the group homology of $G$. 
The conjecture was firstly formulated in \cite{fj} by Farrell and Jones. Then in \cite{DL},
Davis and L\"uck
gave a general framework for the formulations of various isomorphism conjectures in $K$-and $L$-theories. In these formulations, the coefficients are untwisted rings.  Later on, the conjecture was extended by Bartels and Reich \cite{BR} to allow for coefficient in any additive category $\mathcal A$ with a right $G$-action.  This more general version  has better inheritance properties.  For example, the more general version is closed under taking subgroups, finite product of groups and free products of groups. See \cite[Section 2.3]{BFL} for a summary.  The conjecture has many important applications in geometry, topology and algebra. See \cite{LR} for a survey. 

In this paper, we deal with the more general version of the FJC and study the problem that whether the conjecture is closed under passage to over-groups of finite indices. That is,  if $H<G$ is of finite index and $H$ satisfies the FJC with coefficient in every additive category, whether $G$ also satisfies the FJC with coefficient in every additive category.

Our study leads us first to the proposal of the following general question in algebraic $K$- and $L$-theories:

\begin{problem}\label{q1} Let $\mathcal A$ be any additive category with the property that for every finite group $I$, the algebraic $K$-theory of the untwisted group additive category $\mathcal A[I]$ vanishes, i.e. $K_n(\mathcal A[I])=0, \forall n\in\mathbb Z$. Does it follow that for every finite group $F$ with a right action $\alpha$ on $\mathcal A$,  the algebraic 
$K$-theory of the twisted group additive category $\mathcal A_\alpha[F]$ also vanish? What about $L$-theory? 
\end{problem}

\begin{rem} The notions of \textit{group additive category} and \textit{twisted group additive category} are natural generalizations of the notions of group ring and twisted group ring. When a group $G$ acts on an associative ring with unit $R$ from the left via $\alpha$,  then there is a natural induced action, which we still denote by $\alpha$,  of $G$ on the additive category  $R^f_\oplus$ of finitely generated left free $R$-modules from the right. Then the group additive category $(R^f_\oplus)_\alpha[G]$ is naturally equivalent 
to the additive category $\big(R_\alpha[G]\big)^f_\oplus$ of finitely generated left free $R_\alpha[G]$-modules, where $R_\alpha[G]$ is the twisted group ring of $R$ over $G$ by $\alpha$. More details can be found in \cite{BR}. This notion will be recalled in Definition \ref{groupcat},  since it will be important to our treatment. 
\end{rem}

\begin{rem} Similar questions as above  arise and have been studied in various places of mathematics.  For example, it is known that if the \textit{Bass Nil-groups} of a ring $R$ vanish, i.e. $NK_n(R)=0, n\in\mathbb Z$, then $NK_n(R[H])$ is rationally trivial for any finite group $H$.   This was proved by 
Weibel \cite{We} for some special rings and by Hambleton and  L\"uck \cite{HL} for general rings. Recall that the\textit{ Bass Nil-group} $NK_n(R)$ of a ring $R$ is defined to be the kernel of the map $K_n(R[t])\rightarrow K_n(R)$, that is induced by $R[t]\rightarrow R, t\mapsto 0$.
\end{rem}

The reason that we are interested in the above question lies in the following theorem:

\begin{thmA} 
The Farrell-Jones conjecture is closed under 
passage to over-groups of finite indices if Problem \ref{q1} has a positive answer.
\end{thmA}

\begin{rem} In application to our treatment, we only need the additive category $\mathcal A$ in Problem \ref{q1} to be some special category, called the \textit{obstruction category}, that arises in the controlled algebra approach to the FJC. If we replace $\mathcal A$ by such a category in Problem \ref{q1}, then our proof shows that the statement in Theorem A is if and only if.
Because of this, a counterexample to Problem \ref{q1}, though does not give a counterexample to the FJC directly, is then some evidence that there might be some counterexample to the  conjecture. 
\end{rem}


Next we apply induction theory, i.e.  the theory of \textit{Mackey functors}, \textit{Green rings}  and  \textit{Green modules} developed by Swan\cite{SR}, Lam\cite{LT}, Green\cite{Gj} and Dress \cite{DA1}\cite{DA2} to study Problem \ref{q1}. 
Let us firstly recall some terminologies.  Let $p$ be a prime number.  
A finite group is called \textit{$p$-hyperelementary} if it is isomorphic to a group of the form $C\rtimes P$, a semi-direct product of a cyclic group $C$ of order prime to $p$ and a $p$-group $P$ .  If the above semi-direct product is a product, then the group is called \textit{$p$-elementary}.  A group is called \textit{elementary} or \textit{hyperelementary} if it is $p$-elementary or $p$-hyperelementary for some prime number $p$. 
We  denote by $\mathcal H, \mathcal H_p, \mathcal E, \mathcal E_p$  and $\mathcal {FC}$ the class of all hyperelementray, $p$-hyperelementray, elementary, $p$-elementary and finite cyclic groups. 

We also recall the notion of  \textit{Swan ring} $Sw(F)$ associated to a finite group $F$.  It is defined to be  the Grothendieck group of isomorphism classes of all finitely generated left $\mathbb Z[F]$-modules that are projective over $\mathbb Z$, with $[M]=[M']+[M'']$ whenever there is a short exact sequence $0\rightarrow M'\rightarrow M\rightarrow M''\rightarrow 0$ of such modules. Tensor product over $\mathbb Z$ makes $Sw(F)$ into a commutative ring with unit. Note that the unit is given by the class determined by $\mathbb Z$ with the trivial $F$-action.  The class in $Sw(F)$ represented by $\mathbb Z[F]$ is denoted by $[F]$. We denote by $Sw(F)_{[F]}$ the localization of $Sw(F)$ at $[F]$, i.e. we invert $[F]$ in $Sw(F)$.

In the Appendix, we will give a detailed proof of  an induction theorem (Theorem \ref{induction}) for the $K$- and $L$-theories of finite groups with twisted coefficients in additive categories. As an application of this theorem, we have the following:

\vskip 10pt
\begin{thmB} The following statements hold:
\begin{enumerate}
\item The $K$-theory part of Problem \ref{q1} has a positive answer if and only if it has a positive answer for every $F\in\mathcal H$.
\item The rational version, i.e. after tensoring with $\mathbb Q$ over $\mathbb Z$,  of the $K$-theory part of Problem \ref{q1} has a positive answer if and only if this version has a positive answer for every $F\in\mathcal {FC}$. 
\item Under the assumptions of the $K$-theory part of Problem \ref{q1}, the group $K_n(\mathcal A_\alpha[F]), \forall n\in\mathbb Z$,  is a module over the Swan ring $Sw(F)$,  and its localization at $[F]$ vanishes, i.e. 
$$K_n(\mathcal A_\alpha[F])\otimes_{Sw(F)}Sw(F)_{[F]}=0, \forall n\in\mathbb Z$$

The above statements hold for the $L$-theory part of Problem \ref{q1} if we replace $F\in\mathcal H$ by $F\in\mathcal H_2\cup\bigcup_{p\ \text{prime}, p\ne 2}\mathcal E_p$ in (1) and the \textit{Swan ring} Sw$(F)$ by the \textit{Dress ring} Dr$(F)$ in (3) \textup{(}for the definition of Dr$(F)$, see \cite[Page 293-294]{DA2}\textup{)}.

\end{enumerate}
\end{thmB}

\vskip 10pt
The corresponding statements as in Theorem B for the FJC take the following forms:  

\vskip 10pt
\begin{thmC} The following statements hold:
\begin{enumerate}
\item The $K$-theoretic FJC is closed under passage to over-groups of finite indices  if and only if the following holds: for every group $\Gamma$ of the form $\Gamma=G\rtimes_\alpha P$, the semi-direct product of a group $G$ with a hyperelementary group $P$, if $G$ satisfies the conjecture, then $\Gamma$ also satisfies the conjecture.
\item The rational version of the $K$-theoretic FJC is  closed under passage to over-groups of finite indices if and only if the following holds: for every group $\Gamma$ of the form $\Gamma=G\rtimes_\alpha C$, the semi-direct product of a group $G$ with a finite cyclic group $C$, if $G$ satisfies the rational FJC, then $\Gamma$ also satisfies the rational FJC.

\item Let $\Gamma=G\rtimes_\alpha P$ be as in (1) and $\mathcal A$ be an additive category with a right $\Gamma$ action via $\beta$.   Then the groups $H^\Gamma_*(E_{\mathcal{VC}}\Gamma; \textbf{\textup{K}}_{\mathcal A})$, $K_*(\mathcal A_{\beta}[\Gamma])$ as in the assembly map \ref{k assembly},  are modules over the Swan ring $Sw(P)$ and if the $K$-theoretic FJC holds for $G$, then the \textit{localized} assembly map
$$H^\Gamma_*(E_{\mathcal{VC}}\Gamma; \textbf{\textup{K}}_{\mathcal A})\otimes_{Sw(P)}Sw(P)_{([P])}\rightarrow K_n(\mathcal A_\beta[\Gamma])\otimes_{Sw(P)}Sw(P)_{([P])}$$
is an isomorphism.

The above statements hold for the $L$-theoretic FJC if we replace $P\in\mathcal H$ by $P\in\mathcal H_2\cup\bigcup_{p\ \text{prime}, p\ne 2}\mathcal E_p$ in (1) and (3), and the \textit{Swan ring} Sw$(P)$ by the \textit{Dress ring} Dr$(P)$ in (3). 

\end{enumerate}
\end{thmC}

\vskip 10pt
The organization of the paper is simple. In Section \ref{conalg}, we very briefly recall   the formulation of the FJC and an important fact from the controlled algebra approach to the FJC. In Section \ref{main}, we prove our main Theorems. The main ingredients involved in the proofs are \cite[Theorem 3.6]{Wang1} and the induction theorem that is proved in the Appendix. 

\vskip 10pt
\noindent
\textbf{Acknowledgement.}
The author would like to thank his thesis advisor Professor Jean-Fran\c{c}ois Lafont for his enormous and continuous support throughout the years. Part of the project was initiated when the author was a graduate student at the Ohio State University under the guidance of Jean. The author would also like to thank Professor Guoliang Yu at the Texas A\&M university and Shanghai Center for Mathematical Science. The author learned a lot through numerous discussions with Guoliang. The author also thanks Shanghai Center for Mathematical Science for its hospitality during the author's visit in Summer 2015.   Part of the project was done during this visit.



\section{Formulation of the conjecture and the obstruction category}\label{conalg}

In this section, we very briefly recall the formulation of the FJC and the controlled algebra approach to it. More details can be found in \cite{fj},\cite{DL},\cite{BFJR},\cite{BR}.


We first recall the notion of  \textit{group additive category} that is already mentioned in the introduction. 
Let $\mathcal A$ be a right $G$-additive category (we always assume  $\mathcal A$ comes with an involution, which is compatible with the right $G$-action, when we talk about $L$-theory). For every left $G$-set $X$, Bartels and Reich \cite{BR} defined a new additive category $\mathcal A*_GX$. The special case when $X=pt$ will be important
in our treatment and we denote it by $\mathcal A_{\alpha}[G]$, where $\alpha$ denotes the right $G$-action on $\mathcal A$. 

\begin{defn}\label{groupcat}(\cite[Definition 2.1]{BR}) Objects of the category $\mathcal A_{\alpha}[G]$ are the same as the objects of $\mathcal A$. A morphism $\phi: A\rightarrow B$ from $A$ to $B$ in $\mathcal A _\alpha[G]$
is a formal sum $\phi=\sum_{g\in G}\phi^g\cdot g$, where $\phi^g: A\rightarrow g^*B$ is a morphism in $\mathcal A$ and there are only finitely many $g\in G$ with $\phi^g\not=0$. Addition of morphisms is defined in
the obvious way.  Composition of morphism is defined as follows: let $\phi=\sum_{k\in G}\phi^k\cdot k: A\rightarrow B$ be a morphism from $A$ to $B$ and $\psi=\sum_{h\in G}\psi^h\cdot h: B\rightarrow C$ be a morphism
from $B$ to $C$, their composition is given by 
$$\psi\circ\phi:=\sum_{g\in G}\big(\sum_{k,h\in G, g=hk}k^{*}(\psi^h)\circ \phi^k \big)\cdot g$$
\end{defn}

\vskip 10pt

Now let $H^G_*(-; \textbf{K}_{\mathcal A})$ and $H^G_*(-;\textbf{L}^{<-\infty>}_{\mathcal A})$ be the two $G$-equivariant homology theories constructed by Bartels and Reich in \cite{BR}, using the method of Davis and L\"uck \cite{DL}. These two equivariant homology theories have the property that for every subgroup $H<G$, $H^G_n(G/H; \textbf{K}_{\mathcal A})=K_n(\mathcal A_{\alpha}[H]),  H^G_n(G/H;\textbf{L}^{<-\infty>}_{\mathcal A})=L_n^{<-\infty>}(\mathcal A_\alpha[H]), \forall n\in\mathbb Z$. In particular, $H^G_n(pt; \textbf{K}_{\mathcal A})=K_n(\mathcal A_{\alpha}[G])$ and $H^G_n(pt;\textbf{L}^{<-\infty>}_{\mathcal A})=L_n^{<-\infty>}(\mathcal A_\alpha[G])$.

For $\mathcal F$,   a family of subgroups of $G$,  which is closed under taking subgroups and conjugations,  denote by $E_{\mathcal F}G$ a model for the \textit{classifying space} of $G$ relative to the family $\mathcal F$. It is a $G$-CW complex and is characterized, up to
$G$-equivariant homotopy equivalence, by the properties that every isotropy group of the action lies in the family $\mathcal F$, and the fixed point set of $H$ is contractible if $H\in\mathcal F$ and empty if $H\notin\mathcal F$. 
Examples of classifying spaces are $EG, E_{\mathcal {FIN}}G$, and $E_{\mathcal VC}G$, corresponding to the families of trivial, finite, and virtually cyclic  subgroups  of $G$ respectively.

\vskip 10pt
\noindent
\textit{The Farrell-Jones Conjecture.}\label{FJC}
 Let $G$ be a group and $\mathcal A$ be an additive category with a right $G$-action. We say the group $G$ satisfies the $K$-theoretic FJC with coefficient in $\mathcal A$ if the $K$-theoretic assembly map 
\begin{align}\label{k assembly}
& A{_{K}^{\mathcal A}}: H^G_n(E_\mathcal{VC}G; \textbf{K}_{\mathcal A})\rightarrow H^G_n(pt; \textbf K_{\mathcal A})=K_n(\mathcal A_\alpha[G])
\end{align}
which is induced by the obvious map $E_\mathcal{VC}G\rightarrow pt$, is an isomorphism for all $n\in\mathbb Z$. We say the group $G$ satisfies the $L$-theoretic FJC with coefficient in $\mathcal A$ if  the 
corresponding $L$-theoretic assembly map
\begin{align}\label{l assembly}
& A_L^{\mathcal A}: H^G_n(E_\mathcal{VC}G; \textbf{L}^{<-\infty>}_{\mathcal A})\rightarrow H^G_n(pt; \textbf L^{<-\infty>}_{\mathcal A})=L^{<-\infty>}_n(\mathcal A_\alpha[G])
\end{align}
is an isomorphism for all $n\in\mathbb Z$.

\vskip 10pt
We now recall the following important fact from the controlled algebra approach to the FJC. 

\begin{thm} \label{vanishing}\textup{(}\cite{BLR2}\cite{BL2}\textup{)} Let $G$ be a group and $\mathcal A$ be a right $G$-additive category. There exists an additive category $\mathcal O^G(E_\mathcal{VC}G; \mathcal A)$, so that the $K$-theoretic FJC holds for $G$ with coefficient in $\mathcal A$  if and only if the $K$-theory of  $\mathcal O^{G}(E_{\mathcal{VC}}G; \mathcal A)$ vanishes,  i.e.  $K_n(\mathcal O^{G}(E_{\mathcal{VC}}G; \mathcal A))=0, \forall n\in\mathbb Z$. The corresponding  statement holds for the $L$-theoretic FJC.
\end{thm}

\begin{rem} Because of the above theorem, the category $\mathcal O^G(E_{\mathcal {VC}}G; \mathcal A)$ is usually referred to as the \textit{obstruction category}.  It consists of certain \textit{geometric modules} over the space $G\times E_{\mathcal{VC}}G\times[1, \infty)$.  More details about geometric modules and the construction of this category can be found in the references mentioned above.
\end{rem}

\begin{rem} In this paper, when we say a group $G$ satisfies the FJC, we always mean it satisfies the conjecture with coefficients in all additive categories. Otherwise, we will specify the particular coefficient. 
\end{rem}

We will also need the following result from \cite{Wang1}:

\begin{thm}\label{equivalence}\textup{(}\cite[Lemma 3.2, Theorem 3.6]{Wang1}\textup{)} Let $\Gamma=G\rtimes_\alpha F$, where $F$ is a finite group, and  $\mathcal A$ be an additive category with a right $\Gamma$-action. View $\mathcal A$ as a right $G$- and $F$-additive category, and $E_{\mathcal VC}\Gamma$ as a $G$- and $F$-CW complex in the natural way. 
Then there is a right $F$-action on $\mathcal O^G(E_{\mathcal VC}\Gamma; \mathcal A)$, which we still denote by $\alpha$, with the property that the two   categories $\mathcal O^{\Gamma}(E_{\mathcal VC}\Gamma; \mathcal A)$, $\mathcal O^G(E_{\mathcal VC}\Gamma; \mathcal A)_\alpha[F]$ are equivalent as additive categories. 
\end{thm}

\section{Proof of the Main theorems}\label{main}
In this section, we prove our main theorems.  We first need a simple lemma.  

\subsection{A simple Lemma} The following simple lemma is surely well-known to experts.

\begin{lem}\label{classifying} The following statements hold:
\begin{enumerate}
\item Let $H<G$ be a subgroup. Then every model $E_{\mathcal VC}G$ for the classifying space of $G$
relative to the family $\mathcal{VC}$,  viewed as an $H$-CW complex via the inclusion of $H$ in $G$,  is also a model for the classifying space of $H$ relative to $\mathcal {VC}$.

\item Let $G=H\times F$,  where $F$ is a finite group. Then every model $E_{\mathcal{VC}}H$  for the classifying space of $H$ relative to $\mathcal {VC}$, viewed as a $G$-CW complex via the projection of $G$ onto $H$, is also a model for the classifying space of $G$ relative to the family $\mathcal{VC}$.\end{enumerate}
\end{lem}
\begin{proof} (1) is immediate by the characterization properties of a classifying space. Let us prove (2).  Let $\pi: G\rightarrow H$ be the projection of $G$ onto $H$.  Every isotropy group of the action of $G$ on $E_{\mathcal{VC}}H$ has the form $V\times F$, where $V<H$ is an isotropy group of the action of $H$
on $E_{\mathcal{VC}}H$.  Therefore $V$ is virtually cyclic and hence $V\times F$ is virtually cyclic since $F$ is finite. This proves every isotropy group of the $G$-action on $E_{\mathcal{VC}}H$ is virtually cyclic.   Now let $V<G$ be virtually cyclic, we have to show its fixed point set $(E_{\mathcal{VC}}H)^V$ is contractible. But this follows easily from the fact that $(E_{\mathcal{VC}}H)^V=(E_{\mathcal{VC}}H)^{\pi(V)}$.
Now if $V<G$ is not virtually cyclic, then since $F$ is finite, $\pi(V)<H$ is also not virtually cyclic. Therefore $(E_{\mathcal{VC}}H)^V=(E_{\mathcal{VC}}H)^{\pi(V)}=\emptyset$. This completes the proof.
\end{proof}

\noindent
\subsection{Proof of Theorem A} We consider the $K$-theoretic FJC only, the $L$-theoretic case is completely parallel.   Let $H<G$ be a subgroup of finite index and suppose $H$ satisfies the FJC with coefficient in every additive category.  Consider the translation action of $G$ on $G/H$. It induces a group homomorphism $G\rightarrow Aut(G/H)$. Let us denote its kernel by $N$.  Then $N$ is a normal subgroup of $G$ and $N<H$. Note $[G:N]$ is finite since $Aut(G/H)$ is finite. By assumption,   $H$ satisfies the FJC, hence $N$ also satisfies the FJC, since the FJC with coefficients  is closed under taking subgroups.  Now let $F=G/N$ and consider the wreath product
$N\wr F=N^{|F|}\rtimes_{\alpha} F$, where $F$ acts on $N^{|F|}$ by translating the components. Since $N$ satisfies the conjecture, it follows that $N^{|F|}$ also satisfies the conjecture,  since the conjecture is closed under taking finite direct sum.  Now since $G$ injects into $N^{|F|}\rtimes_{\alpha} F$, see for example \cite[Section 2]{FR}, in order to show $G$ satisfies the FJC, it suffices to show $N^{|F|}\rtimes_{\alpha} F$
satisfies the FJC. 

Now let  $Q=N^{|F|}$, $\Gamma= Q\rtimes_\alpha F$ and $\mathcal A$ be any additive category with a right  $\Gamma$-action.  By the first part of Lemma \ref{classifying},  $E_{\mathcal {VC}}\Gamma$ is also a model for $E_{\mathcal{VC}}Q$.  Therefore by the second part of Lemma \ref{classifying}, $E_{\mathcal{VC}}\Gamma$ is also a model for $E_{\mathcal{VC}}(Q\times I)$ for every finite group $I$, where $Q\times I$ acts on $E_{\mathcal{VC}}\Gamma$ via the projection of $Q\times I$ to $Q$. In particular, $I$ acts on $E_{\mathcal{VC}}\Gamma$ trivially. Endow $\mathcal A$ with the action of $Q\times I$ via the projection of $Q\times I$ to $Q$ as well. Then according to the defining formulas in \cite[Lemma 3.2]{Wang1}, the trivial actions of $I$ on $Q$, $E_{\mathcal VC}\Gamma$ and $\mathcal A$ induce the trivial action of $I$ on the obstruction category $\mathcal O^Q(E_{\mathcal{VC}}\Gamma; \mathcal A)$ of $Q$. Therefore, by Theorem \ref{equivalence}, the obstruction category $\mathcal O^{Q\times I}(E_{\mathcal{VC}}\Gamma; \mathcal A)$ of $Q\times I$  is equivalent to the group additive category $\mathcal O^Q(E_{\mathcal{VC}}\Gamma; \mathcal A)[I]$ of the obstruction category $\mathcal O^Q(E_{\mathcal{VC}}\Gamma; \mathcal A)$ of $Q$ over $I$ with the trivial action. 

Now  since $Q$ satisfies the FJC and $I$ is finite, $Q\times I$ also satisfies the FJC. This implies, by Theorem \ref{vanishing},  the $K$-theory of $\mathcal O^{Q\times I}(E_{\mathcal{VC}}\Gamma; \mathcal A)$ vanishes. Therefore, the $K$-theory of  $\mathcal O^Q(E_{\mathcal{VC}}\Gamma; \mathcal A)[I]$ also vanishes for every finite group $I$.  Now assume Problem \ref{q1} has a positive answer. Then for every finite group $J$ and every right action $\beta$ of $J$ on $\mathcal O^Q(E_{\mathcal{VC}}\Gamma; \mathcal A)$, the $K$-theory of the twisted group additive category $\mathcal O^Q(E_{\mathcal{VC}}\Gamma; \mathcal A)_\beta[J]$ also vanishes.  Now according to Theorem \ref{equivalence}, the obstruction category 
$\mathcal O^\Gamma(E_{\mathcal{VC}}\Gamma; \mathcal A)$ of $\Gamma=Q\rtimes_\alpha F$ is equivalent to the twisted group additive category $\mathcal O^Q(E_{\mathcal{VC}}\Gamma; \mathcal A)_\alpha[F]$. By the above argument, the $K$-theory of $\mathcal O^Q(E_{\mathcal{VC}}\Gamma; \mathcal A)_\alpha[F]$ vanishes, therefore the $K$-theory of $\mathcal O^\Gamma(E_{\mathcal{VC}}\Gamma; \mathcal A)$ also vanishes. This implies the $K$-theoretic FJC holds for $\Gamma$. Therefore the $K$-theoretic FJC holds for $G$. This completes the proof of Theorem A. 

\subsection{Proof of Theorem B} Parts (1) and (2) follow directly from the induction theorem,  see Theorem \ref{induction} of the Appendix. Part (3) follows from the fact that $K_n(\mathcal A_\alpha[-])$ is a \textit{Green module} over the \textit{Green ring} $Sw(-)$, see the proof of Theorem \ref{induction} for these notions. More precisely, for every subgroup $I<F$, the group $K_n(\mathcal A_\alpha[I])$ is a module over the Swan ring $Sw(I)$,  and we also have the Frobenius reciprocity law \ref{frobenius2}:
$$ Ind_I^J(x\cdot Res^J_I(y))=Ind_I^J(x)\cdot y,\ \text{for any}\ x\in Sw(I), y\in K_n(\mathcal A_\alpha[J])$$
where $I<J$ are subgroups of $F$, and 
$Ind_I^J, Res_I^J$ are the induction and restriction maps on the corresponding $K$-groups or the Swan rings.

Now in the above identity, we take $I$ to be the trivial group $\{1\}$,  $J$ to be $F$ and $x\in Sw(\{1\})\cong\mathbb Z$ to be the identity, we then get that
$$ Ind_{\{1\}}^F(Res^F_{\{1\}}(y))=Ind_{\{1\}}^F(1)\cdot y,\ \forall y\in K_n(\mathcal A_\alpha[F])$$
Note that the left hand side of the above identity equals zero  since  $Res^F_{\{1\}}(y)\in K_n(\mathcal A)=0$
by assumption.  Note also that $Ind_{\{1\}}^{F}=[F]$, where $[F]\in Sw(F)$ is the class represented by $\mathbb Z[F]$.  Therefore
$$[F]\cdot y=0, \forall y\in K_n(\mathcal A_\alpha[F])$$
Hence the statement of Part (3) holds for the $K$-theory part. The $L$-theory part can be treated similarly.  This completes the proof of Theorem B.

\subsection{Proof of Theorem C} (1) The only if part is easy. Let us show the if part. Let $H<G$ be a subgroup of finite index. Suppose $H$ satisfies the FJC, we want to show $G$ satisfies the FJC. The first paragraph in the proof of Theorem A shows that $G$ injects into a group of the form $\Gamma=Q\rtimes_\alpha F$, where $Q$ satisfies the FJC and $F$ is a finite group. Therefore, it suffices to show $\Gamma$
satisfies the FJC.  By Theorem \ref{equivalence}, the obstruction category $\mathcal O^{\Gamma}(E_{\mathcal{VC}}\Gamma; \mathcal A)$ of $\Gamma$ is equivalent to the group additive category
$\mathcal O^Q(E_{\mathcal{VC}}\Gamma; \mathcal A)_\alpha[F]$.  Now by Theorem \ref{induction}, the $K$-theory of $\mathcal O^Q(E_{\mathcal{VC}}\Gamma; \mathcal A)_\alpha[F]$ vanishes if and only if the $K$-theory of $\mathcal O^Q(E_{\mathcal{VC}}\Gamma; \mathcal A)_\alpha[P]$ vanishes for every hyperelementary subgroup $P<F$, which is true by the fact that  $\mathcal O^Q(E_{\mathcal{VC}}\Gamma; \mathcal A)_\alpha[P]$ is equivalent to $\mathcal O^{Q\rtimes_\alpha P}(E_{\mathcal{VC}}\Gamma; \mathcal A)$ and  the assumption that $Q\rtimes_\alpha P$ satisfies the FJC whenever $Q$ does.  This shows that
the $K$-theory of  $\mathcal O^{\Gamma}(E_{\mathcal{VC}}\Gamma; \mathcal A)$ vanishes. Therefore $\Gamma$ satisfies the FJC by Theorem \ref{vanishing}. Hence $G$ satisfies the FJC. This completes the proof of part (1). 

(2) The proof of part (2) is completely parallel to the proof of part (1). Therefore we omit the detail.

(3) On the obstruction category level, the proof of part (3) is parallel to the proof of part (3) of Theorem B. 
This then passes down to the assembly map level, which is just the statement of part (3). One has to notice that 
the left hand side of the assembly map, i.e. $H^{G\rtimes -}_n(E_{\mathcal{VC}}\Gamma; \textbf{K}_{\mathcal A})$, is a Green module over the Green ring $Sw(-)$ when they are viewed as functors from the category of subgroups of $P$ to the category of abelian groups.  This can be seen on the categorical level using the controlled algebra approach to the FJC, just as we did for the obstruction category, i.e. 
$K_n\big(\mathcal O^{G\rtimes_\alpha -}(E_{\mathcal{VC}}\Gamma; \mathcal A)\big)\cong K_n\big(\mathcal O^{G}(E_{\mathcal{VC}}\Gamma; \mathcal A)_\alpha[-]\big)$ is a Green module over the Green  ring $Sw(-)$ on the category of subgroups of $P$.  This completes the proof of Theorem C.

\section{Appendix}\label{appendix}  In this appendix, we give a detailed proof of an induction theorem for $K_n(\mathcal A_\alpha[F])$  and $L_n^{<i>}(\mathcal A_\alpha[F])$
$i=2, 1, \cdots, -\infty$, where $\mathcal A$ is an additive category with a right action by the finite group $F$ via $\alpha$, and $L_n^{<i>}$ denotes the $L$-group with decoration $i$.
 The corresponding induction theorem for $K_n(R[F])$ and $L_n^{<i>}(R[F])$ is classical, where $R$ is an associative ring with unit and $F$ acts on $R$ trivially. 
For the more general case here, such an induction theorem is also well-known to experts. Indeed, it is already used in the proof of  \cite[Theorem 8.2]{BFL}. However,  a detailed proof  is not given there and does not exist in the literature to the knowledge of the author.  Therefore, we carry out a detailed proof here for record and for future use. 

Although the main idea of proof of the following theorem  is similar to the classical case, extra verifications and cautions are needed  in the current setting. This is mainly due to the fact that, unlike in the classical setting where all  operations can be easily defined, some constructions  in this more general categorical setting involve certain choices.  We will only deal with the $K$-theory part, the $L$-theory part is completely parallel.

\begin{thm}\label{induction} Let $F$ be a finite group and $\mathcal A$ be an additive category. Suppose $F$ acts on $\mathcal A$ from the right via $\alpha$.  Then for every $n\in\mathbb Z$, every prime number $p$ and every  $i\in\{-\infty, \cdots, -1, 0, 1, 2\}$, there are isomorphisms of algebraic $K$- and $L$-groups:
\begin{enumerate}
\item $$\ds\lim_{\substack{\longrightarrow \\ P\in\mathcal H(F)}}K_n(\mathcal A_\alpha[P])\cong K_n(\mathcal A_\alpha[F])\cong\ds\lim_{\substack{\longleftarrow \\ P\in\mathcal H(F)}}K_{n}(\mathcal A_\alpha[P])$$
$$\ds\lim_{\substack{\longrightarrow \\ P\in\mathcal H_p(F)}}K_n(\mathcal A_\alpha[P])\otimes_{\mathbb Z}\mathbb Z_{(p)}\cong K_n(\mathcal A_\alpha[F])\otimes_{\mathbb Z}\mathbb Z_{(p)}\cong\ds\lim_{\substack{\longleftarrow \\ P\in\mathcal H_p(F)}}K_{n}(\mathcal A_\alpha[P])\otimes_{\mathbb Z}\mathbb Z_{(p)}$$
$$\ds\lim_{\substack{\longrightarrow \\ P\in\mathcal {FC}(F)}}K_n(\mathcal A_\alpha[P]){\otimes}_{\mathbb Z}\mathbb Q\cong K_n(\mathcal A_\alpha[F])\otimes_{\mathbb Z}\mathbb Q\cong\ds\lim_{\substack{\longleftarrow \\ P\in\mathcal {FC}(F)}}K_{n}(\mathcal A_\alpha[P])\otimes_{\mathbb Z}\mathbb Q.$$

\item $$\ds\lim_{\substack{\longrightarrow \\ P\in\mathcal H_2(F)\cup\bigcup_{p\ne 2}\mathcal E_p(P)}}L^{<i>}_n(\mathcal A_\alpha[P])\cong L_n^{<i>}(\mathcal A_\alpha[F])\cong\ds\lim_{\substack{\longleftarrow \\ P\in\mathcal H_2(F)\cup\bigcup_{p\ne 2}\mathcal E_p(P)}}L_{n}^{<i>}(\mathcal A_\alpha[P])$$
$$\ds\lim_{\substack{\longrightarrow \\ P\in\mathcal H_2(F)}}L_n^{<i>}(\mathcal A_\alpha[P])\otimes_{\mathbb Z}\mathbb Z_{(2)}\cong L_n^{<i>}(\mathcal A_\alpha[F])\otimes_{\mathbb Z}\mathbb Z_{(2)}\cong\ds\lim_{\substack{\longleftarrow \\ P\in\mathcal H_2(F)}}L_{n}^{<i>}(\mathcal A_\alpha[P])\otimes_{\mathbb Z}\mathbb Z_{(2)}$$
$$\ds\lim_{\substack{\longrightarrow \\ P\in\bigcup_{p\ne 2}\mathcal E_p(P)}}L_n^{<i>}(\mathcal A_\alpha[P])\otimes_{\mathbb Z}\mathbb Z[1/2]\cong L_n^{<i>}(\mathcal A_\alpha[F])\otimes_{\mathbb Z}\mathbb Z[1/2]\cong\ds\lim_{\substack{\longleftarrow \\ P\in\bigcup_{p\ne 2}\mathcal E_p(P)}}L_{n}^{<i>}(\mathcal A_\alpha[P])\otimes_{\mathbb Z}\mathbb Z[1/2]$$
$$\ds\lim_{\substack{\longrightarrow \\ P\in\mathcal {FC}(F)}}L_n^{<i>}(\mathcal A_\alpha[P]){\otimes}_{\mathbb Z}\mathbb Q\cong L_n^{<i>}(\mathcal A_\alpha[F])\otimes_{\mathbb Z}\mathbb Q\cong\ds\lim_{\substack{\longleftarrow \\ P\in\mathcal {FC}(F)}}L_{n}^{<i>}(\mathcal A_\alpha[P])\otimes_{\mathbb Z}\mathbb Q.$$

\end{enumerate}
\vskip 10pt
where  $\mathcal H(F), \mathcal H_p(F), \mathcal {FC}(F), \mathcal E_p(F)$ denote the sets of all hyperelementray, $p$-hyperelementary, finite cyclic and $p$-elementary subgroups of $F$ respectively; $\mathbb Z_{(p)}$ denotes the localization of $\mathbb Z$ at the prime $p$;  and the limits and colimits are taking with respect to all maps induced by inclusions, or by conjugations on subgroups of  $F$.
\end{thm}

\begin{rem} The Farrell-Jones conjecture can be viewed as a generalization of the above induction theorem to the infinite group case. 
\end{rem}

\begin{proof}

The main idea of the proof is to show for all $n\in\mathbb Z$,  the assignment 
$$I\longrightarrow K_n(\mathcal A_\alpha[I])$$
where $I$ runs over all subgroups of $F$, defines a \textit{Mackey functor} and it is a \textit{Green Module} over the \textit{Green functor}
$$I\longrightarrow Sw(I)$$
of the Swan rings of subgroups of $F$. Then the general induction theory of  Mackey functors, Green functors and Green modules imply the isomorphisms above.  See \cite{SR},\cite{LT},\cite{Gj},\cite{DA1},\cite{DA2} for original references of the theory. We also refer the reader to \cite{Webb},\cite[Chapter 11]{Oliver} for a friendly introduction. There are two equivalent definitions of Mackey functors, see \cite[Section 2]{Webb}.
We adopt the first one as given in \cite[Section 2]{Webb} since it is more concrete. 

In order to show $I\longrightarrow K_n(\mathcal A_\alpha(I))$ defines a Mackey functor, we have to construct for any inclusion $I\subseteq J$ of subgroups of $F$, an induction map
$$Ind_I^J:  K_n(\mathcal A_\alpha [I])\longrightarrow K_n(\mathcal A_\alpha[J])$$
and a restriction map
$$Res_I^J:   K_n(\mathcal A_\alpha [J])\longrightarrow K_n(\mathcal A_\alpha[I])$$
and for any $f\in F, I<F$, a conjugation map
$$c_f:    K_n(\mathcal A_\alpha [I])\longrightarrow K_n(\mathcal A_\alpha[fIf^{-1}])$$
that satisfy the defining properties of a Mackey functor. 
These three maps will be the induced maps on $K$-groups of  three additive functors, which we still denote by $Ind_I^J$,
$Res_I^J$ and $c_f$
$$Ind_I^J:  \mathcal A_\alpha [I]\longrightarrow \mathcal A_\alpha[J]$$
$$Res_I^J: \mathcal A_\alpha [J]\longrightarrow \mathcal A_\alpha[I]$$
$$c_f:    \mathcal A_\alpha [I]\longrightarrow \mathcal A_\alpha[fIf^{-1}]$$
that we are going to construct now.

The induction functor $Ind_I^J$ is just the inclusion of $\mathcal A_\alpha[I]$ into $\mathcal A_\alpha[J]$. The restriction functor  will depend on the choice of a complete set of representatives of the left coset spaces $J/I$, but we will show that the induced maps on $K$-groups is independent of such a choice. So  let $r=|J/I|$ and  $\{a_1, a_2,\cdots, \ a_r\}\subseteq J$ be a complete set of representatives of the left coset spaces $J/I$.  We define 
$$Res_I^J: \mathcal A_\alpha [J]\longrightarrow \mathcal A_\alpha[I]$$
as follows.

On objects:  for any object $A$ in $\mathcal A$,  we define  
$$Res^J_I(A):=\displaystyle\bigoplus_{\lambda=1}^ra_\lambda^*A$$

On morphisms: for any morphism $\phi=\displaystyle\sum_{j\in J}\phi^j\cdot j: A\rightarrow B$, where $\phi^j: A\rightarrow j^*B$,   we define $$Res^J_I(\phi):=\displaystyle\sum_{i\in I}Res_I^J(\phi)^i\cdot i:\  \displaystyle\bigoplus_{\lambda=1}^ra_\lambda^*A\longrightarrow \displaystyle\bigoplus_{\tau=1}^ra_\tau^*B$$
by defining 
$$Res_I^J(\phi)^i:\ \displaystyle\bigoplus_{\lambda=1}^ra_\lambda^*A\longrightarrow \displaystyle\bigoplus_{\tau=1}^ri^*a_\tau^*B=\bigoplus_{\tau=1}^r(a_\tau i)^*B$$
to be the map whose $(\lambda, \tau)$-component  is the map
$$a_\lambda^*(\phi^{a_\tau ia_\lambda^{-1}}):\ a_\lambda^*A\longrightarrow(a_\tau i)^*B$$

We have to show $Res_I^J$ is a genuine additive functor.  If $\phi$ is the identity morphism, then $\phi^j=0$ unless $j=e$, for which $\phi^e=Id$.  Therefore $\phi^{a_{\tau}ia_{\lambda}^{-1}}=0$ unless $a_\tau ia_{\lambda}^{-1}=e$, for which we must have $i=e$ and $\tau=\lambda$.
Hence $Res_I^J(\phi)^i=0$ unless $i=e$ for which $Res_I^J(\phi)^i=Id$. Thus $Res_I^J(\phi)$ is the identity morphism if $\phi$ is.  

Now let $$\phi=\displaystyle\sum_{j\in J}\phi^j\cdot j: A\rightarrow B, \ \ \psi=\displaystyle\sum_{j\in J}\psi^j\cdot j: B\rightarrow C$$ be two morphisms in $\mathcal A_\alpha[J]$. 
For any $i\in I$,  the $(\lambda, \tau)$-component of $Res_I^J(\psi\circ\phi)^i$ is given by
$$a_\lambda^*(\psi\circ\phi)^{a_\tau ia_\lambda^{-1}}:\ a_\lambda^*A\longrightarrow(a_\tau i)^*C$$
which is equal to
$$a_\lambda^*\Big(\displaystyle\sum_{s, t\in J, st=a_\tau i a_\lambda^{-1}}t^*(\psi^s)\circ\phi^t\Big)=\displaystyle\sum_{s, t\in J, st=a_\tau i a_\lambda^{-1}}a_\lambda^*t^*(\psi^s)\circ a_\lambda^*(\phi^t)$$
and the $(\lambda, \tau)$-component of $\big(Res_I^J(\psi)\circ Res_I^J(\phi)\big)^i$ is given by 
$$\displaystyle\sum_{\sigma=1}^r\displaystyle\sum_{k,l\in I, kl=i}l^*a_\sigma^*(\psi^{a_\tau k a_\sigma^{-1}})\circ a^*_\lambda(\phi^{a_\sigma l a_\lambda^{-1}})$$

Now in the above sum, we let $t=a_\sigma l a_\lambda^{-1}$.  Then as $\sigma$ runs from $1$ to $r$ and $l$ runs over $I$, $t$ runs over $J$. Also note that $a_\tau k a_\sigma^{-1}=a_\tau ia_\lambda^{-1}t^{-1}$.  Hence the above sum becomes

$$\displaystyle\sum_{t\in J}a_\lambda^*t^*(\psi^{a_\tau i a_\lambda^{-1}t^{-1}})\circ a^*_\lambda(\phi^{t})=\displaystyle\sum_{s, t\in J, st=a_\tau i a_\lambda^{-1}}a_\lambda^*t^*(\psi^s)\circ a_\lambda^*(\phi^t)$$
which is the $(\lambda, \tau)$-component of $Res_I^J(\psi\circ\phi)^i$. This shows $Res_I^J(\psi\circ\phi)=Res_I^J(\psi)\circ Res_I^J(\phi)$ and thus $Res_I^J$ is a genuine functor. Finally the additivity of $Res_I^J$ is trivial. 

Next we show the induced map of $Res_I^J$ on the $K$-groups is independent of the choice of $\{a_1, a_2, \cdots, a_r\}\subseteq J$. So let $\{a_1i_1, a_2i_2, \cdots, a_ri_r\}$ be another complete set of representatives, where $i_1, i_2, \cdots, i_r\in I$. We denote the corresponding restriction functor by $\overline{Res}_I^J$ and show that it is naturally isomorphic to the functor $Res_I^J$. So let $A$ be any object in $\mathcal A$. We define a morphism 
$$\xi_A=\sum_{i\in I}\xi_A^i\cdot i:\  Res_I^J(A)=\displaystyle\bigoplus_{\lambda=1}^ra_\lambda^*A\longrightarrow \overline{Res}_I^J(A)=\displaystyle\bigoplus_{\tau=1}^r(a_\tau i_\tau)^*A$$
by defining 
$$\xi_A^i: \ \displaystyle\bigoplus_{\lambda=1}^ra_\lambda^*A\longrightarrow \displaystyle\bigoplus_{\tau=1}^ri^*(a_\tau i_\tau)^*A=\displaystyle\bigoplus_{\tau=1}^r(a_\tau i_\tau i)^*A $$
to be the map whose $(\lambda, \tau)$-component
$$\xi^i_{A, \lambda, \tau}:\ a_\lambda^*A\longrightarrow (a_\tau i_\tau i)^*A$$ 
is the identity if $a_\lambda=a_\tau i_\tau i$  (which can happen if and only if $\lambda=\tau$ and $i_\tau=i^{-1}$) and to be $0$ otherwise. It's easy to see that $\xi_A$ is an isomorphism in $\mathcal A_\alpha[I]$. We show it indeed is a
natural isomorphism between  $Res_I^J$ and $\overline{Res}_I^J$. 

So let 
$$\phi=\displaystyle\sum_{j\in J}\phi^j\cdot j: A\longrightarrow B$$
be any morphism in $\mathcal A_\alpha[J]$.  For any $i\in I$,  the $(\lambda, \tau)$-component of $(\overline{Res}_I^J(\phi)\circ\xi_A)^i$ is given by
$$\displaystyle\sum_{\sigma=1}^r\sum_{s, t\in I, st=i}(a_\sigma i_\sigma t)^*(\phi^{a_\tau i_\tau si_\sigma^{-1}a_\sigma^{-1}})\circ\xi^t_{A, \lambda, \sigma}$$
whose only non-zero  summand is when $\sigma=\lambda$, $t=i_\lambda^{-1}$ and  the above sum simplifies to 
$$a_\lambda^*(\phi^{a_\tau i_\tau i a_{\lambda}^{-1}})$$
 The  $(\lambda, \tau)$-component of $\big(\xi_B\circ Res_I^J(\phi)\big)^i$ is given by
$$\displaystyle\sum_{\sigma=1}^r\sum_{s, t\in I, st=i}t^*(\xi^s_{B, \sigma, \tau})\circ a_\lambda^*(\phi^{a_\sigma ta_\lambda^{-1}})$$
whose only non-zero summand is when $\sigma=\tau$ and $s=i_\tau^{-1}$. In this case $t=i_\tau i$. Hence the above sum also simplifies to
$$a_\lambda^*(\phi^{a_\tau i_\tau i a_{\lambda}^{-1}})$$
This shows $\overline{Res}_I^J(\phi)\circ\xi_A=\xi_B\circ Res_I^J(\phi)$ and thus $Res_I^J$ and $\overline{Res}_I^J$ are naturally isomorphic. Therefore they induce the same map on $K$-groups. Hence the induced map on $K$-groups is independent of the choice of the complete set of representatives. 

We now turn into the construction of the conjugation functor
$$c_f:    \mathcal A_\alpha [I]\longrightarrow \mathcal A_\alpha[fIf^{-1}]$$

On objects: $c_f(A):=(f^{-1})^*A$ for any object $A$ in $\mathcal A$.

On morphisms: let $\phi=\displaystyle\sum_{i\in I}\phi^i\cdot i:\ A\longrightarrow B$ be a morphism in $\mathcal A_\alpha[I]$, where $\phi^i: A\longrightarrow i^*B$ is a morphism in $\mathcal A$.  Define 
$$c_f(\phi)=\displaystyle\sum_{i\in I}c_f(\phi)^{fif^{-1}}\cdot fif^{-1}: (f^{-1})^*A\longrightarrow (f^{-1})^*B$$
by defining 
$$c_f(\phi)^{fif^{-1}}:\ (f^{-1})^*A\longrightarrow (fif^{-1})^*(f^{-1})^*B=(f^{-1})^*i^*B$$
to be $(f^{-1})^*(\phi^i)$. It is not hard to verify $c_f$ is a genuine additive functor.  Hence it induces a map on $K$-groups.

We now show the above three maps on $K$-groups satisfy the defining properties of a Mackey functor.

(1) \textit{$Ind_I^I, Res_I^I, c_f, $ are the identity maps for all $I<F, f\in I$:} this is trivial for $Ind_I^I$ and $Res_I^I$. In order to show $c_f, f\in I$ is the identity map, we show the corresponding functor 
$$c_f:    \mathcal A_\alpha [I]\longrightarrow \mathcal A_\alpha[I]$$
is naturally isomorphic to the identity functor.  So for any object $A$ in $\mathcal A$, we define 
$$\eta_A=\displaystyle\sum_{i\in I}\eta_A^i\cdot i:\ A\longrightarrow (f^{-1})^*(A) $$
be defining $$\eta_A^i:\ A\longrightarrow i^{*}(f^{-1})^*(A)=(f^{-1}i)^*(A)$$
to be zero if $i\ne f$ and to be the identity if $i=f$. This is clearly an isomorphism in $\mathcal A_\alpha[I]$. It's also not hard to show
$c_f(\phi)\circ\eta_A=\eta_B\circ\phi$. Hence $c_f$ is naturally isomorphic to the identity functor. Thus it induces the identity on $K$-groups.
\vskip 5pt
(2) \textit{$Res_I^J\circ Res_J^K=Res_I^K$ for all $I<J<K<F$}: let $\{a_1, \cdots, a_r\}$ be a complete set of representatives for $K/J$ and $\{b_1, \cdots, b_s\}$ be a complete set of representatives for $J/I$. Then $\{a_1b_1, a_1b_2, \cdots, a_rb_s\}$ is a complete set of representatives for $K/I$. Then the corresponding restriction functors satisfy $Res_I^J\circ Res_J^K=Res_I^K$. Since the induced maps on $K$-groups are independent of the choices of the complete sets of representatives, we have $Res_I^J\circ Res_J^K=Res_I^K$ on $K$-groups. 
\vskip 5pt
(3) \textit{$Ind_J^K\circ Ind_I^J=Ind_I^K$ for all $I<J<K<F$}:  this is trivial by definition.
\vskip 5pt
(4) $c_f\circ c_g=c_{fg}$ for all $f, g\in F$: this is trivial by definition.
\vskip 5pt
(5) $Res_{fIf^{-1}}^{fJf^{-1}}\circ c_f=c_f\circ Res_I^J$ for all $I<J$ and $f\in F$: this is easy by definition.
\vskip 5pt
(6) $Ind_{fIf^{-1}}^{fJf^{-1}}\circ c_f=c_f\circ Ind_I^J$  for all $I<J$ and $f\in F$: this is trivial by definition.
\vskip 5pt
(7) The double coset formula holds, i.e. 
$$Res_J^K\circ Ind_I^K=\displaystyle\sum_{[f]\in J\backslash K/I }Ind_{J\cap fIf^{-1}}^{J}\circ c_f\circ Res_{f^{-1}Jf\cap I}^I$$
 for all $I<K, J<K$:  let $\{f_1, \cdots, f_r\}\subseteq K$ be a set of complete representatives for the double coset space $J\backslash K/I$, then $\{f_1^{-1}, \cdots, f_r^{-1}\}$ is a complete set of representatives for $I\backslash K/J$. For each $\lambda\in\{1, \cdots, r\}$, let $\{a_{\lambda 1}, \cdots, a_{\lambda s_\lambda}\}\subseteq I$ be a complete set of representatives for $I/(f_\lambda^{-1}Jf_\lambda\cap I)$, then 
$\{a_{\lambda 1}f_\lambda^{-1}J, \cdots, a_{\lambda s_\lambda}f_\lambda^{-1}J\}\subseteq K/J$ is the orbit of $f_\lambda^{-1}J$ under the left action of $I$ on $K/J$. Hence $\bigcup_{\lambda=1}^r\{a_{\lambda 1}f_\lambda^{-1}, \cdots, a_{\lambda s_\lambda}f_\lambda^{-1}\}\subseteq K$ is a complete set of representatives for $K/J$.  Using these complete sets of representatives, it is then easy to see that on the level of categories, we have the equality of functors: 
$$Res_J^K\circ Ind_I^K=\displaystyle\sum_{[f]\in J\backslash K/I }Ind_{J\cap fIf^{-1}}^{J}\circ c_f\circ Res_{f^{-1}Jf\cap I}^I$$
Therefore by the additivity theorem of algebraic $K$-theory \cite[Page 401]{We1}, the above equality holds on the level of $K$-groups. 

We now have completed the proof that for all $n\in\mathbb Z$,  $K_n(\mathcal A_\alpha[-])$ is a Mackey functor on the subgroups of $F$. Next we show it is a Green module over the Green functor $Sw(-)$ of Swan rings of subgroups of $F$. In order to do this, it is convenient to introduce the following. 
Let $I<F$ be any subgroup. Define $\mathcal F(\mathbb Z[I])$ to be the exact category whose objects are $(\mathbb Z^m, \beta)$, where $m\ge 0$ is an integer and $\beta: I\longrightarrow GL_m(\mathbb Z)$ is an $I$-action on $\mathbb Z^m$. We denote $\beta(i)$  by  $M_{\beta(i)}$ to indicate it is a matrix.  The morphism set from $(\mathbb Z^{m_1}, \beta_1)$ to $(\mathbb Z^{m_2}, \beta_2)$ consists of  all $m_2\times m_1$ matrices $M$ with integer entries so that $MM_{\beta_1(i)}=M_{\beta_2(i)}M, \forall i\in I$.  Compositions of morphisms are given by multiplications of matrices. A sequence 
$$0\longrightarrow (\mathbb Z^{m_1}, \beta_1)\longrightarrow (\mathbb Z^{m}, \beta)\longrightarrow (\mathbb Z^{m_2}, \beta_2)\longrightarrow 0$$
of morphisms in $\mathcal F(\mathbb Z[I])$ is said to be \textit{exact} if the corresponding sequence 
$$0\longrightarrow \mathbb Z^{m_1}\longrightarrow \mathbb Z^{m}\longrightarrow \mathbb Z^{m_2}\longrightarrow 0$$
of abelian groups is exact.  This defines an exact structure on $\mathcal F(\mathbb Z[I])$. Note that the operation $(\mathbb Z^{m_1}, \beta_1)\otimes(\mathbb Z^{m_2}, \beta_2):=(\mathbb Z^{m_1m_2}, \beta_1\otimes\beta_2)$, where $M_{(\beta_1\otimes\beta_2)(i)}=M_{\beta_1(i)}\otimes M_{\beta_2(i)}\in GL_{m_1m_2}(\mathbb Z)$ defines
a ``tensor product" in $\mathcal F(\mathbb Z[I])$.    Now the Swan ring $Sw(I)$ can be identified with the Grothendieck group of $\mathcal F(\mathbb Z[I])$ and the ring structure on $Sw(I)$ is induced by the above tensor product. It is well-known that $Sw(-)$ is a Green functor on subgroups of $F$. 

We now construct a bi-exact functor
$$\Theta: \mathcal F(\mathbb Z[I])\times \mathcal A_\alpha[I]\longrightarrow \mathcal A_\alpha[I]$$
as follows (the following construction depends on choices of direct sums in $\mathcal A_\alpha[I]$, which is not a problem):

On objects: $\Theta((\mathbb Z^m, \beta), A)=A^m$. 

On morphisms: let $M: (\mathbb Z^{m_1}, \beta_1)\longrightarrow (\mathbb Z^{m_2}, \beta_2)$ and $\phi=\displaystyle\sum_{i\in I}\phi^i\cdot i: A\longrightarrow B$ be morphisms in $\mathcal F(\mathbb Z[I])$ and  $\mathcal A_\alpha[I]$ respectively,  where $\phi^i: A\longrightarrow i^*B$ is a morphism in $\mathcal A$. We define 

 $$\Theta(M, \phi):=\displaystyle\sum_{i\in I}MM_{\beta_1(i)}\phi^i\cdot i:  A^{m_1}\longrightarrow B^{m_2}$$
where $MM_{\beta_1(i)}\phi^i$ is the matrix obtained by multiplying $\phi^i$ to the entries of the matrix $MM_{\beta_1(i)}$.  Note that we have $MM_{\beta_1(i)}\phi^i=M_{\beta_2(i)}M\phi^i$. 

We check that $\Theta$ is a genuine bi-exact functor. Clearly $\Theta(Id, Id)=Id$. Now let 
$$M_1: (\mathbb Z^{m_1}, \beta_1)\longrightarrow (\mathbb Z^{m_2}, \beta_2),\ \ \ \ M_2: (\mathbb Z^{m_2}, \beta_2)\longrightarrow (\mathbb Z^{m_3}, \beta_3)$$
be two morphisms in $\mathcal F(\mathbb Z[I])$ and 
$$\phi=\displaystyle\sum_{i\in I}\phi^i\cdot i: A\rightarrow B, \ \ \ \ \psi=\displaystyle\sum_{i\in I}\psi^i\cdot i: B\rightarrow C$$
 be two morphisms in $\mathcal A_\alpha[I]$.
Then 
$$\Theta\big((M_2M_1, \psi\circ\phi)\big)=\displaystyle\sum_{i\in I}M_2M_1M_{\beta_1(i)}(\psi\circ\phi)^i\cdot i$$
and 
\begin{align*}
\Theta(M_2, \psi)\circ\Theta(M_1, \phi) &=\displaystyle\sum_{i\in I}\Big(\displaystyle\sum_{s, t\in I, st=i}t^*(M_2M_{\beta_2(s)}\psi^s)\circ(M_1M_{\beta_1(t)}\phi^t)\Big)\cdot i\\
&=\displaystyle\sum_{i\in I}\Big(\displaystyle\sum_{s, t\in I, st=i}M_2M_1M_{\beta_1(s)}M_{\beta_1(t)}\big((t^*\psi^s)\circ\phi^t\big)\Big)\cdot i\\
&=\displaystyle\sum_{i\in I}M_2M_1M_{\beta_1(i)}\big(\displaystyle\sum_{s, t\in I, st=i}(t^*\psi^s)\circ\phi^t\big)\cdot i\\
&=\displaystyle\sum_{i\in I}M_2M_1M_{\beta_1(i)}(\psi\circ\phi)^i\cdot i\\
\end{align*}
This shows $\Theta\big((M_2M_1, \psi\circ\phi)\big)=\Theta\big((M_2M_1, \psi\circ\phi)\big)$ and therefore $\Theta$ is a genuine bi-functor. 

For any object $(\mathbb Z^m, \alpha)$ of $\mathcal F(\mathbb Z[I])$, the functor  $\Theta((\mathbb Z^m, \alpha), -)$ is additive, hence it is exact since the exact structure on $\mathcal A_\alpha[I]$ is split exact.   Now for any object $A$ in $\mathcal A_\alpha[I]$ and any short exact sequence (possibly not split exact)
$$
\xymatrix{
0\ar[r] &(\mathbb Z^{m_1}, \beta_1) \ar[r]^-{M_1} & (\mathbb Z^{m}, \beta)\ar[r]^{M_2} & (\mathbb Z^{m_2}, \beta_2)\ar[r] & 0 \\
}$$
in $\mathcal F(\mathbb Z[I])$. The sequence in $\mathcal A_\alpha[I]$ obtained by applying the functor $\Theta(-,  A)$ to the above sequence is 
$$
\xymatrix{
0\ar[r] & A^{m_1}\ar[r]^-{M_1} & A^{m_2} \ar[r]^{M_2}& A^{m_3}\ar[r] & 0 \\
}$$
which is split exact since the short exact sequence
$$
\xymatrix{
0\ar[r] & \mathbb Z^{m_1}\ar[r]^-{M_1} & \mathbb Z^{m_2} \ar[r]^{M_2}& \mathbb Z^{m_3}\ar[r] & 0 \\
}$$
of free abelian groups splits.  Therefore the functor $\Theta(-, A)$ is also exact. Hence the bi-functor $\Theta$
is bi-exact.  Therefore, by the additivity theorem of algebraic $K$-theory,  $\Theta$ induces a pairing, which we still denote by $\Theta$, of abelian groups
$$\Theta: Sw(I)\times K_n(\mathcal A_\alpha[I])\longrightarrow K_n(\mathcal A_\alpha[I])$$
Note that as a bi-exact functor, $\Theta$ has the property  that 
$$\Theta((\mathbb Z^{m_1}, \beta_1)\otimes(\mathbb Z^{m_2}, \beta_2), -)=\Theta((\mathbb Z^{m_1}, \beta_1), -)\circ\Theta((\mathbb Z^{m_2}, \beta_2), -)$$
Therefore the above paring defines a $Sw(I)$-module structure on $K_n(\mathcal A_\alpha[I])$. For brevity, we denote $\Theta(x, y)$ by $x\cdot y$ for any $x\in Sw(I), y\in K_n(\mathcal A_\alpha[I])$. 

Now in order to show $\Theta$ defines a Green module structure on the Mackey functor $K_n(\mathcal A_\alpha[-])$ over the Green functor $Sw(-)$, we have to show the \textit{Frobenius reciprocity laws} hold, i.e. we have to show for any inclusion $I\subseteq J$ of subgroups of $F$, the following identities hold:
\begin{align}\label{frobenius1}
 Ind_I^J(Res^J_I(x)\cdot y)=x\cdot Ind_I^J(y), \text{for any}\ x\in Sw(J), y\in K_n(\mathcal A_\alpha[I]).\\
\label{frobenius2}
 Ind_I^J(x\cdot Res^J_I(y))=Ind_I^J(x)\cdot y, \text{for any}\ x\in Sw(I), y\in K_n(\mathcal A_\alpha[J]).
\end{align}

%

The identity \ref{frobenius1} follows directly from the fact that the following compositions of functors equal:
$$
\xymatrix{
\mathcal F(\mathbb Z[J])\times \mathcal A_\alpha[I] \ar[rr]^{Res^J_I\times Id} &&\mathcal F(\mathbb Z[I])\times\mathcal A_\alpha[I]\ar[r]^{\ \ \ \ \ \ \Theta} & \mathcal A_\alpha[I]\ar[r]^{Ind_I^J}& \mathcal A_\alpha[J]
}
$$
$$
\xymatrix{
\mathcal F(\mathbb Z[J])\times \mathcal A_\alpha[I] \ar[rr]^{Id\times Ind^J_I} &&\mathcal F(\mathbb Z[J])\times \mathcal A_\alpha[J]\ar[r]^{\ \ \ \ \ \ \Theta} &\mathcal A_\alpha[J] 
}
$$
 where $$Res_I^J: \mathcal F(\mathbb Z[J])\longrightarrow \mathcal F(\mathbb Z[I])$$
is the obvious restriction functor. 

The identity \ref{frobenius2} requires a little more work, since it involves the induction map
$$Ind_I^J: Sw(I)\longrightarrow Sw(J)$$ 
which is induced by an induction functor 
$$Ind_I^J: \mathcal F(\mathbb Z[I])\longrightarrow \mathcal F(\mathbb Z[J])$$
However, the definition of the induction functor requires a choice of a complete set of representatives for $J/I$. Note however that the restriction functor in \ref{frobenius2}  also involves such a choice. If we choose them to be the same, then the second equality follows easily. We omit the details. 

We now have completed the verification work that the Mackey functor $K_n(\mathcal A_\alpha[-])$ is a Green module over the Green ring $Sw(-)$.  Now the isomorphisms as stated in the theorem follow from the general theory of induction, see for example \cite[Theorem 11.1]{Oliver}, and the corresponding induction theorem for the Swan ring $Sw(F)$, see \cite[Corollary 4.2]{SR}, \cite[proof of Lemma 4.1]{BL1}.

For the $L$-theory case, we use the Dress ring $Dr(F)$ instead of the Swan ring $Sw(F)$. One can similarly define an action of the Green ring $Dr(-)$ on the Mackey functors $L^{<i>}_n(\mathcal A_\alpha[-])$. Then the general theory of induction and the  induction theorem for $Dr(-)$, see \cite[Theorem 1]{DA2}, imply the corresponding isomorphisms for $L$-groups. This completes the proof of the theorem.
\end{proof}

\frenchspacing

\bibliographystyle{plain}
\bibliography{kun}

\end{document}